\def\PaperDate{2008/06/30}
\documentclass[12pt,letterpaper]{article}
  \usepackage{datum}
  \usepackage[GlobalNumbered]{buxmath}
  \usepackage{graphicx}
  \usepackage{refstyle}
  \usepackage[american]{babel}
  \usepackage{crypt}
  
  \InputIfFileExists{basic_notation.tex}{}{}

  \newvariable{\Log}{\operatorname{log}}

  \newvariable{\TheGroup}{G}
  \newvariable{\TheGroupElement}{u}
  \newvariable{\AltGroupElement}{v}
  \newvariable{\TheLength}{m}
  \newvariable{\AltLength}{n}
  \newvariable{\TheExponent}{\delta}
  \newvariable{\AltExponent}{\varepsilon}
  \newvariable{\TheLastIndex}{t}
  \newvariable{\AltLastIndex}{s}

  \newvariable{\TheLhs}{u}
  \newvariable{\TheRhs}{w}
  \newvariable{\TheConjElement}{x}
  \newvariable{\TheNumEq}{r}

  \newvariable{\ThF}{F}
  \newvariable{\ThEffGen}{x}

  \newvariable{\TheNumWords}{r}
  \newvariable{\TheWordLength}{n}
  \newvariable{\TheWidth}{w}
  \newvariable{\TheIndex}{i}
  \newvariable{\AltIndex}{j}
  \newvariable{\TheRank}{n}
  \newvariable{\TheNumStrands}{n}
  \newvariable{\TheHeight}{h}
  \newvariable{\TheProjection}{\pi}
  \newvariable{\TheInjection}{\iota}
  \newvariable{\BrownGeoghegan}{Y}
  \newvariable{\BGcover}{\tilde{\BrownGeoghegan}}
  \newvariable{\bgequiv}{\sim}

  \newvariable{\GenIndex}{n}
  \newvariable{\HighIndex}{m}
  \newvariable{\LowIndex}{q}

  \newvariable{\BV}{BV}
  \newvariable{\BF}{BF}
  \newvariable{\ThEff}{F}
  \newvariable{\ThVee}{V}
  \newvariable{\BraidGroup}{B}
  \newvariable{\TheBvElement}{g}
  \newvariable{\AltBvElement}{g'}

  \newvariable{\TheDiagram}{\Delta}
  \newvariable{\AltDiagram}{\Theta}
  \newvariable{\ThrDiagram}{\Lambda}

  \newvariable{\TheTree}{T}
  \newvariable{\TheBraid}{\beta}
  \newcommand{\Top}{\mathrm{top}}
  \newcommand{\Bot}{\mathrm{bot}}
  \newvariable{\TheWord}{w}

  \newvariable{\EffGen}{\nu}
  \newvariable{\BrdGen}{\bar{\pi}}
  \newvariable{\AltGen}{\pi}

  \newcommand{\BigO}[1]{\Oka\left(#1\right)}

  \newvariable{\TheNormalClosure}{N}

  \newvariable{\TheTrivialElement}{1}
  \newvariable{\Eps}{\varepsilon}

  \newvariable{\AliceSize}{m}
  \newvariable{\AliceLength}{k}
  \newvariable{\AliceElement}{a}
  \newvariable{\AliceKey}{a}
  \newvariable{\AliceIndex}{i}
  \newvariable{\AliceExp}{\delta}
  \newvariable{\BobSize}{n}
  \newvariable{\BobLength}{l}
  \newvariable{\BobElement}{b}
  \newvariable{\BobKey}{b}
  \newvariable{\BobIndex}{j}
  \newvariable{\BobExp}{\varepsilon}
  \newvariable{\TheConjugator}{c}

\begin{document}
  \title{Some Remarks on the Braided Thompson Group $\BV$}
  \author{Kai-Uwe~Bux \and Dmitriy~Sonkin}
  \date{\datum\PaperDate}
  \maketitle
  \begin{abstract}
    M.\,Brin and P.\,Dehornoy independently discovered a braided
    version $\BV$ of R.\,Thompson's group $\ThVee$.
    In this paper, we discuss some properties of
    $\BV$ that might make the group interesting for group
    based cryptography. In particular, we show that $\BV$ does
    not admit a non-trivial linear representation.
  \end{abstract}

  \section{Introduction}\label{sec:introduction}

    One of the ways to visualize elements of R.\,Thompson's group
    $\ThEff$ is to regard them as pairs of
    trees~\cite{Cannon.Floyd.Parry:1996}.
    The trees forming such a pair, called the top tree
    and the bottom
    tree, are finite binary trees with the same number of
    leaves. We follow \cite{Belk.Brown:2005} in drawing the
    top tree with the root at the top and the bottom tree with
    the root at its bottom aligning their leaves to match.
    An element of Thompson's group $\ThVee$ can be understood
    in a similar way: we still have a pair of trees, but now we wedge
    a permutation in between that decides which leaves are considered
    matching.

    The braided version $\BV$ of Thompson's group $\ThVee$ was introduced
    independently by Brin in \cite{Brin:2007}, \cite{Brin:2006} and
    Dehornoy
    in \cite{Dehornoy:2005} and has been investigated
    further by several authors \cite{Brady.Burillo.Cleary.Stein:2008},
    \cite{Burillo.Cleary:2007}. Informally speaking, one obtains an element
    of the braided Thompson's group $\BV$ by using a braid instead
    of a permutation to connect the leaves of the top tree to the
    leaves of the bottom tree. In Section~\ref{sec:complexity},
    we discuss complexity issues of computations
    in $\BV$. In particular,
    we show that multiplication of two elements of $\BV$ given in
    tree-braid-tree form can be carried out in quadratic time on the
    input length. In Section~\ref{sec:representations}, we
    analyze Brin's presentation of $\BV$ to prove the following:
    \begin{NewTh}<statement>{Theorem}
      The group $\BV$ does not admit non-trivial linear representations
      in any characteristic.
    \end{NewTh}

    We note that relatives of $\BV$, namely braid groups and
    Thompson's group $\ThEff$, received some attention recently
    from a cryptographic point of view. Section~\ref{sec:cryptography}
    reflects on the possibility of using the group $\BV$ as a
    platform group in cryptographic protocols.

  \section{The Group \boldmath$\BV$ and
  its Braided Band Diagrams}\label{sec:diagrams}

    Recall that elements of Thompson's group $\ThEff$ can be
    represented by \notion{band diagrams}. A band diagram encodes
    splitting and merging of a band keeping track of the relative order
    of splits and merges. Pictorially, one can think of band
    diagrams as thickened tree diagrams. The following picture shows
    band diagrams for the canonical generators
    $\ThEffGen[\Zero]$ and $\ThEffGen[\One]$:
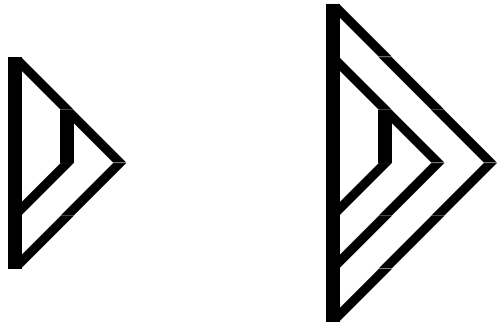
    \[
      \begin{bv}
        \xs\bvline
        \xl\xs\bvline
        \xm\xdl\bvline
        \xm\bvline
      \end{bv}
      \qquad\qquad\qquad
      \begin{bv}
        \xs\bvline
        \xs\xdr\bvline
        \xl\xs\xdr\bvline
        \xm\xdl\xdl\bvline
        \xm\xdl\bvline
        \xm\bvline
      \end{bv}
    \]
    Two band diagrams are equivalent if one can pass from one to the other
    by means of a finite sequence of moves, where each move applies
    (forward or backward) one of the following relations:
    
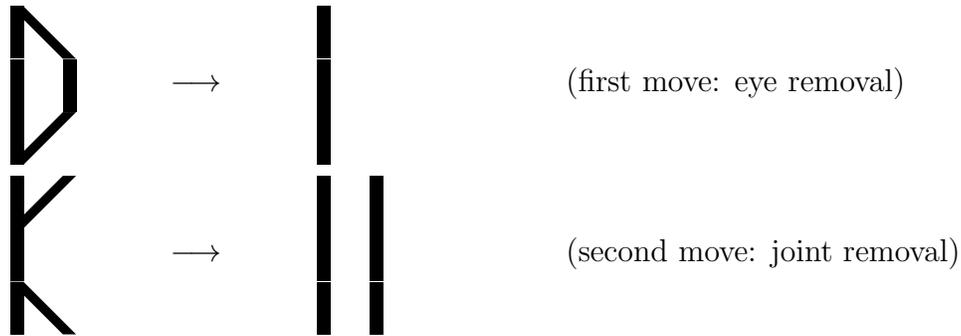
\begin{eqnarray*}
      \begin{bv}
        \xs\bvline
        \xl\xl\bvline
        \xm\bvline
      \end{bv}\kern-7mm
      &\qquad\longrightarrow\qquad&
      \begin{bv}
        \xl\xsp\bvline
        \xl\bvline
        \xl\bvline
      \end{bv}
      \qquad\qquad\text{(first move: eye removal)}
      \\
      \begin{bv}
        \xm\bvline
        \xl\bvline
        \xs\bvline
      \end{bv}
      &\qquad\longrightarrow\qquad&
      \begin{bv}
        \xl\xl\bvline
        \xl\xl\bvline
        \xl\xl\bvline
      \end{bv}
      \qquad\qquad\text{(second move: joint removal)}
    \end{eqnarray*}
    A band diagram is called \notion{reduced} if neither of the above
    relations can be applied forward. It is well known that every band
    diagram can be reduced by a finite sequence of forward applications
    of the relations and that every equivalence class of band diagrams
    has a unique reduced representative.

    Elements of Thompson's group $\ThEff$ correspond to equivalence
    classes of band diagrams. Multiplication of elements of $\ThEff$
    translates into stacking band diagrams.

    Allowing bands to braid, one arrives at the notion
    of \notion{braided band diagrams}. Those represent elements
    of the group $\BV$. Note that bands are allowed to braid, but
    they are not allowed to twist, i.e., a twisted band segment
    like
    \[
      \begin{array}{c}
        \loadimage{crypt.100}
      \end{array}
    \]
    is not allowed
    in a braided band diagram.

    Also note that we do not distinguish diagrams that just differ
    in the way the braiding is drawn (i.e., the diagrams themselves
    are supposed to live in $3$-space and are regarded equal if they
    differ by an ambient homotopy not twisting bands). E.g., the
    following two pictures describe the same diagram:
    
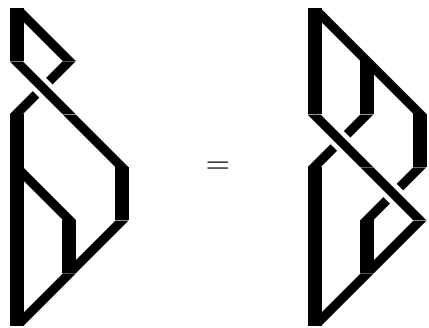
\[
      \begin{bv}
        \xs\bvline
        \xo\bvline
        \xl\xdr\bvline
        \xs\xr\bvline
        \xl\xm\bvline
        \xm\bvline
      \end{bv}
      \qquad
      =
      \qquad
      \begin{bv}
        \xs\bvline
        \xl\xs\bvline
        \xo\xr\bvline
        \xl\xo\bvline
        \xl\xm\bvline
        \xm\bvline
      \end{bv}
    \]

    Again, two diagrams are \notion{equivalent} if there is a
    finite sequence of moves transforming one into the other; and
    we call a diagram \notion{reduced} if it does not allow for a
    forward application of a relation.

    M.\,Brin \cite[Theorem~2]{Brin:2006} has shown that
    $\BV$ is generated by the following elements:
    
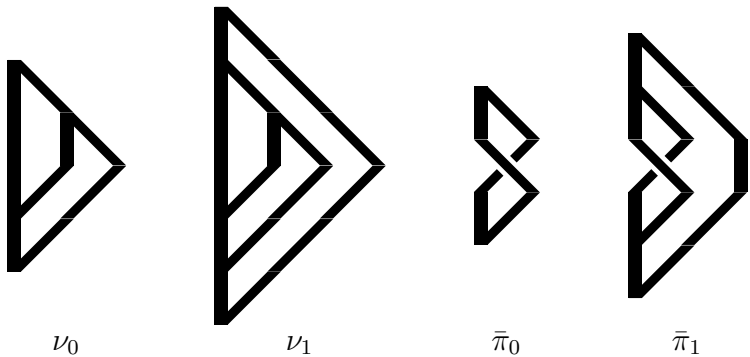
\[
      \begin{array}{c@{\kern 1cm}c@{\kern 1cm}c@{\kern 1cm}c}
        \begin{bv}
          \xs\bvline
          \xl\xs\bvline
          \xm\xdl\bvline
          \xm\bvline
        \end{bv}
        &
        \begin{bv}
          \xs\bvline
          \xs\xdr\bvline
          \xl\xs\xdr\bvline
          \xm\xdl\xdl\bvline
          \xm\xdl\bvline
          \xm\bvline
        \end{bv}
        &
        \begin{bv}
          \xs\bvline
          \xo\bvline
          \xm\bvline
        \end{bv}
        &
        \begin{bv}
          \xs\bvline
          \xs\xdr\bvline
          \xo\xr\bvline
          \xm\xdl\bvline
          \xm\bvline
        \end{bv}
        \\
        \EffGen[\Zero]
        &
        \EffGen[\One]
        &
        \BrdGen[\Zero]
        &
        \BrdGen[\One]
      \end{array}
    \]

    \begin{prop}
      Every equivalence class of braided band diagrams
      contains a unique reduced representative, and this representative
      can be obtained from any diagram in the equivalence class
      via a finite sequence of forward moves.
    \end{prop}
    \begin{proof*}
      Let $\TheDiagram$ and $\AltDiagram$ be two braided band
      diagrams. We write
      \(
        \TheDiagram\rightarrow\AltDiagram
      \)
      if there is a forward move from $\TheDiagram$ to
      $\AltDiagram$. Since forward moves decrease the number of
      band-segments in a diagram, it follows that
      ``$\rightarrow$'' is a \notion{noetherian} relation,
      i.e., there are no infinite $\rightarrow$\,-chains.

      By Newman's Lemma (a standard result on rewriting systems;
      see, e.g., \cite[Corollary~4.76]{Becker.Weispfenning:1993}),
      it suffices to show that the $\rightarrow$\,-relation
      is \notion{locally confluent}, i.e., given a diagram
      $\TheDiagram$ and two forward moves
      \(
        \TheDiagram\rightarrow\AltDiagram[\One]
      \)
      and
      \(
        \TheDiagram\rightarrow\AltDiagram[\Two]
        ,
      \)
      there exists a diagram $\ThrDiagram$ that
      can be obtained by forward move sequences from both
      $\AltDiagram[\One]$ and $\AltDiagram[\Two]$.

      The local confluence condition, however, is easily verified
      in our setting:
      \begin{enumerate}
        \item
          Any two forward moves removing eyes (joints) can be
          performed in any order since the two eyes (joints) do
          not interfere with each other.
        \item
          Given two forward moves of different type, either they
          can be performed in any order, or they lead to equal
          diagrams (possibly after a suitable ambient
          homotopy). The latter happens when an eye meets a
          joint (removing either of them yields a tripod).

          In the following example, we either
          delete the top-eye or the following joint and
          obtain identical diagrams (i.e., diagrams that are equal
          after a suitable ambient homotopy):
          
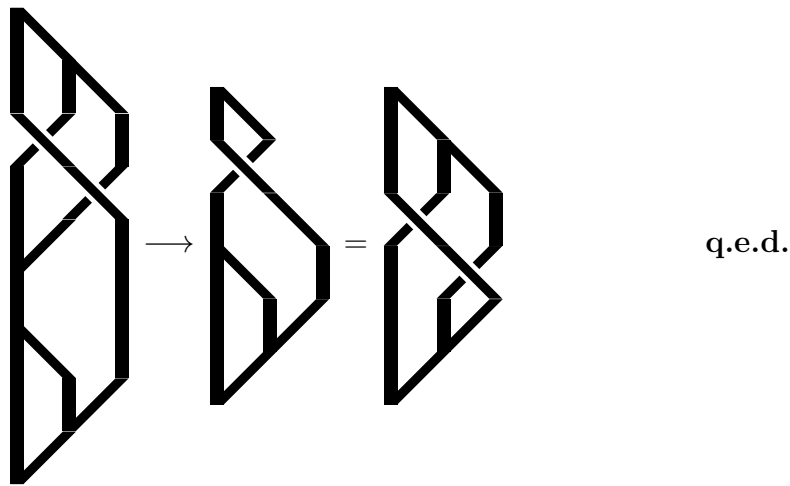
\[
            \begin{bv}
              \xs\bvline
              \xl\xs\bvline
              \xo\xr\bvline
              \xl\xo\bvline
              \xm\xr\bvline
              \xl\xr\bvline
              \xs\xr\bvline
              \xl\xm\bvline
              \xm\bvline
            \end{bv}
            \longrightarrow
            \begin{bv}
              \xs\bvline
              \xo\bvline
              \xl\xdr\bvline
              \xs\xr\bvline
              \xl\xm\bvline
              \xm\bvline
            \end{bv}
            =
            \begin{bv}
              \xs\bvline
              \xl\xs\bvline
              \xo\xr\bvline
              \xl\xo\bvline
              \xl\xm\bvline
              \xm\bvline
            \end{bv}
            \qed
          \]
      \end{enumerate}
    \end{proof*}

    \begin{rem}
      Note that, as a corollary, we recover the result of
      M.\,Brin \cite[Lemma~4.3]{Brin:2006} that $\BV$ contains
      a copy of $\ThEff$ realized as the set of reduced diagrams that
      do not exhibit braiding.
    \end{rem}

  \section{Complexity of the Word
  Problem}\label{sec:complexity}

    We want to devise an efficient method for computing products in
    $\BV$. To do so, we have to establish a canonical method of
    representing elements of $\BV$ in a way suitable for
    computations. Braided band diagrams will serve as our starting
    point.

    Let us call a diagram \notion{semi-reduced} if it does not
    admit joint-removal moves. Obviously, every
    reduced diagram is semi-reduced. Moreover, every semi-reduced
    diagram can be transformed into a reduced diagram via
    a (finite) sequence of eye-removal moves.

    \begin{observation}
      A diagram $\TheDiagram$ is semi-reduced if and only if, along
      each route from top to bottom in $\TheDiagram$, we never find a merge
      of bands followed by a split of the band.\qed
    \end{observation}
    \begin{observation}
      Consider a semi-reduced braided band diagram $\TheDiagram$. We
      can isotop the diagram so that all the splits precede any
      braiding and all the merges occur after all the braiding is
      done:
      
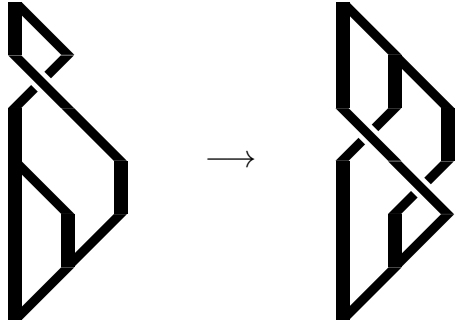
\[
        \begin{bv}
          \xs\bvline
          \xo\bvline
          \xl\xdr\bvline
          \xs\xr\bvline
          \xl\xm\bvline
          \xm\bvline
        \end{bv}
        \qquad\longrightarrow\qquad
        \begin{bv}
          \xs\bvline
          \xl\xs\bvline
          \xo\xr\bvline
          \xl\xo\bvline
          \xl\xm\bvline
          \xm\bvline
        \end{bv}
      \]
      Thus, a semi-reduced diagram always decomposes into three
      layers: the top-part that is a root-at-the-top tree where all
      the splits of the band occur; the middle part consisting of
      a braid of bands; the bottom part which is a root-at-the-bottom
      tree where the bands are merged back into a single ribbon.

      Consequently, every element of $\BV$ can be represented by
      a triple
      \(
        \TupelOf{\TheTree[][\Top],\TheBraid,\TheTree[][\Bot]}
        ,
      \)
      consisting of two planar trees $\TheTree[][\Top]$ and
      $\TheTree[][\Bot]$ and a braid $\TheBraid$ interpolating
      between the leaves of the trees.\qed
    \end{observation}
    \begin{observation}
      Conversely, given a triple
      \(
        \TupelOf{\TheTree[][\Top],\TheBraid,\TheTree[][\Bot]}
      \)
      as above, we can form a braided band diagram by stacking the
      top tree on the top of the braid and appending an
      upside-down drawing of the bottom tree. Within such a diagram,
      along each ribbon we find no merge followed by a split,
      i.e., the diagram is semi-reduced.\qed
    \end{observation}

    We shall now discuss how to detect removable eyes. Let the
    triple
    \(
      \TupelOf{\TheTree[][\Top],\TheBraid,\TheTree[][\Bot]}
    \)
    represent a semi-reduced diagram. Assuming that the braid
    $\TheBraid$ is an element of the braid
    group $\BraidGroup[\TheRank]$, where $\TheRank$ is the number
    of leaves of either tree, let
    \[
      \TheProjection[\TheIndex] \mapcolon
      \BraidGroup[\TheRank] \longrightarrow \BraidGroup[\TheRank-\One]
    \]
    be the map defined by deleting the $\TheIndex^{\text{th}}$ strand
    (strands are indexed at the top of the braid); and let
    \[
      \TheInjection[\TheIndex] \mapcolon
      \BraidGroup[\TheRank-\One] \longrightarrow \BraidGroup[\TheRank]
    \]
    be the map defined by doubling the $\TheIndex^{\text{th}}$ strand
    (i.e., splitting that strand into two all the way from the top to
    the bottom of the braid).

    \begin{observation}
      Let $\TheBraid\in\BraidGroup[\TheRank]$ be a braid. The
      $\TheIndex^{\text{th}}$ and
      $(\TheIndex+1)^{\text{st}}$
      strands are parallel, i.e., can be united into a single strand without
      otherwise disrupting the braid $\TheBraid$, if and only if
      \(
        \TheInjectionOf[\TheIndex]{
          \TheProjectionOf[\TheIndex]{\TheBraid}
        }
        =
        \TheBraid
        .
      \)\qed
    \end{observation}
    \begin{observation}\label{obs:eye_removal}
      A semi-reduced diagram represented as a triple
      \(
        \TupelOf{\TheTree[][\Top],\TheBraid,\TheTree[][\Bot]}
      \)
      can be further reduced if and only if there is a pair of
      parallel strands in $\TheBraid$ that connects a terminal
      caret in $\TheTree[][\Top]$ to a terminal caret in
      $\TheTree[][\Bot]$. Here, a
      \notion{terminal caret}
      in $\TheTree[][\Top]$ is a split of a band such that
      along both resulting bands there are no further splits.
      Symmetrically, a terminal caret in $\TheTree[][\Bot]$ is a
      merge of two bands both of which had not previously been
      involved in merges.\qed
    \end{observation}

    We can use this to reduce diagrams algorithmically.
    \begin{NewTh}<definition>[Satz]{Algorithm}\label{alg:reduce}
      A triple
      \(
        \TupelOf{
          \TheTree[\One][\Top],
          \TheBraid[\One],
          \TheTree[\One][\Bot]
        }
      \)
      can be reduced by applying a sequence of
      eye-removal moves according to Observation~\ref{obs:eye_removal}.
      The process can be organized as follows:
      \begin{enumerate}
        \item
          Find the left-most terminal caret of the top tree.
        \item
          Check whether the strands issuing from this caret are parallel.
          If so, check whether they lead to a terminal caret in the
          bottom tree. If so, remove the eye and check if there
          is a terminal caret in the current position
          (in the top tree). Repeat this step, if there is one.
        \item
          Move to the right and repeat the previous step on the next
          terminal caret in the top tree.
        \item
          Repeat until all terminal carets of the top tree have
          been visited.
      \end{enumerate}
      In this algorithm, we can proceed from the left to
      the right since an eye-removal cannot create terminal carets
      in the top tree to the left of the caret that is being removed.
    \end{NewTh}
    Checking whether two triples represent the same group element
    in $\BV$ can be performed according to the following:
    \begin{NewTh}<definition>[Satz]{Algorithm}\label{alg:test_equal}
      Given two triples
      \(
        \TupelOf{\TheTree[\One][\Top],\TheBraid[\One],\TheTree[\One][\Bot]}
      \)
      and
      \(
        \TupelOf{\TheTree[\Two][\Top],\TheBraid[\Two],\TheTree[\Two][\Bot]}
        ,
      \)
      perform a sequence of eye-removal moves on either of them until
      both cannot be further reduced. The triples thus obtained represent the same group element
      if and only if they have the same top and bottom trees and the
      braids are equal as elements of the corresponding braid group.
    \end{NewTh}

    Multiplication also has a natural interpretation in terms of
    diagrams:
    \begin{observation}\label{obs:product}
      If two elements $\TheBvElement[\One]$ and $\TheBvElement[\Two]$
      are represented by triples
      \(
        \TupelOf{\TheTree[\One][\Top],\TheBraid[\One],\TheTree[\One][\Bot]}
      \)
      and
      \(
        \TupelOf{\TheTree[\Two][\Top],\TheBraid[\Two],\TheTree[\Two][\Bot]}
      \)
      where $\TheTree[\Two][\Top]=\TheTree[\One][\Bot]$, then the
      triple
      \(
        \TupelOf{
          \TheTree[\One][\Top],
          \TheBraid[\One]\TheBraid[\Two],
          \TheTree[\Two][\Bot]
        }
      \)
      represents the product $\TheBvElement[\One]\TheBvElement[\Two]$.\qed
    \end{observation}
    Consequently, multiplication in $\BV$ can be carried out
    using the following:
    \begin{NewTh}<definition>[Satz]{Algorithm}\label{alg:product}
      Given two elements
      \(
        \TheBvElement[\One],
        \TheBvElement[\Two]
        \in
        \BV,
      \)
      represented by semi-reduced triples
      \(
        \TupelOf{\TheTree[\One][\Top],\TheBraid[\One],\TheTree[\One][\Bot]}
      \)
      and
      \(
        \TupelOf{\TheTree[\Two][\Top],\TheBraid[\Two],\TheTree[\Two][\Bot]}
        ,
      \)
      compute a semi-reduced triple for the product
      \(
        \TheBvElement[\One]\TheBvElement[\Two]
      \)
      as follows: first unreduce both factors so that the bottom
      tree of the left-hand factor matches the top tree of the
      right-hand factor; then form a triple for the product using
      Observation~\ref{obs:product}. Note that the resulting
      triple is automatically semi-reduced.
    \end{NewTh}

    So far, we have ignored complexity issues and we have taken
    operations on braids and trees for granted. Since braid operations
    dominate the time complexity of all algorithms, we will not
    discuss the complexity of operations on trees.

    To meaningfully discuss the time complexity of the
    algorithms above, we need to
    settle on a representation of the braid component of a
    triple. The braid is an element of the braid group
    $\BraidGroup[\TheRank]$ where the number $\TheRank$
    of strands is determined by the tree components of the triple.
    A natural way to represent elements of $\BraidGroup[\TheRank]$
    is as words over some fixed generating set. We will be using
    the set of non-repeating braids (also called the
    Garside generators). For this set of generators,
    W.\,Thurston has given a solution to the word problem in braid
    groups \cite[Chapter~9]{Epstein_et_al}.

    Recall that a braid $\TheBraid\in\BraidGroup[\TheRank]$ is \notion{positive}
    if it can be drawn so that all crossings are overcrossings
    (the down-right strand goes over the
    down-left strand).
    A positive braid is called \notion{non-repeating} if any
    pair of strands crosses at most once. By
    \cite[Lemma~9.1.10]{Epstein_et_al}, non-repeating braids of
    $\BraidGroup[\TheRank]$ are uniquely
    determined by the permutation they induce; and for each
    permutation, there is a non-repeating braid. Thus,
    non-repeating braids form a generating set for $\BraidGroup[\TheRank]$
    whose elements can be represented by permutations on $\TheRank$
    letters.

    The following observation makes the set of non-repeating
    braids convenient for our purposes:
    \begin{observation}\label{obs:cap_word_length}
      Neither doubling a strand nor deleting a strand creates
      undercrossings out of nowhere. Also, both operations do not
      increase the number of crossings of any given pair of strands.
      Thus, if $\TheBraid$ is a non-repeating braid, then so are
      \(
        \TheProjectionOf[\TheIndex]{\TheBraid}
      \)
      and
      \(
        \TheInjectionOf[\TheIndex]{\TheBraid}
      \)
      for any $\TheIndex$.

      It follows that the operations of deleting and doubling strands
      do not increase the word length with respect to the generating
      set of non-repeating braids.
    \end{observation}

    We also note that non-repeating braids can be manipulated efficiently:
    the operations of doubling a strand or deleting a strand in a
    generator are linear in the length of the input and, therefore,
    take time $\BigO{\TheRank\LogOf{\TheRank}}$ in the case of
    a non-repeating braid of $\BraidGroup[\TheRank]$.

    For the generating set of non-repeating braids,
    Thurston defines the right-greedy and the left-greedy normal
    forms, which are unique and can be efficiently
    computed:
    \begin{lemma}[{\cite[Corollary~9.5.3]{Epstein_et_al}}]
      Let a braid $\TheBraid\in\BraidGroup[\TheRank]$ be a word
      of length $\TheHeight$ with respect to the generating set
      of non-repeating braids. Then $\TheBraid$ can be put in either
      normal form in time
      \(
        \BigO{\TheHeight[][\Two]\TheRank\LogOf{\TheRank}}
        .
      \)\qed
    \end{lemma}

    For computations in $\BV$, we use the right-greedy
    normal form.
    \begin{Def}
      The \notion{normal form} of an element of $\BV$ is a
      triple
      \(
        \TupelOf{\TheTree[][\Top],\TheWord,\TheTree[][\Bot]}
        ,
      \)
      where $\TheWord$ is a word over the generating set of
      non-repeating braids in right-greedy normal form so that
      the diagram represented by the triple is reduced.
      (Of course, the way such a triple represents a diagram
      is by regarding the word as representing a braid.)
    \end{Def}
    \begin{prop}\label{prop:reduce}
      Any triple
      \(
        \TupelOf{\TheTree[][\Top],\TheWord,\TheTree[][\Bot]}
        ,
      \)
      where the trees have $\TheRank$ leaves and $\TheWord$ is
      of length $\TheHeight$, can be put into normal
      form in time
      \(
        \BigO{\TheHeight[][\Two]\TheRank[][\Two]\LogOf{\TheRank}}
        .
      \)
    \end{prop}
    \begin{proof}
      Since a tree with $\TheRank$ leaves has at most $\TheRank$
      carets, Algorithm~\ref{alg:reduce} requires at most $\TheRank$
      unsuccessful checks for eyes and at most $\TheRank$ successful
      checks. Each check can be carried out with complexity
      \(
        \BigO{\TheHeight[][\Two]\TheRank\LogOf{\TheRank}}
        .
      \)
      Removing an eye that has been found is done by computing
      $\TheProjectionOf[\TheIndex]{\TheWord}$ for the corresponding
      $\TheIndex$. This is done for each generator in the expression
      of $\TheWord$; and thus, it is linear in $\TheHeight$. Thus,
      we can eliminate a single eye in
      \(
        \BigO{\TheHeight\TheRank\LogOf{\TheRank}}
      \)
      time.

      Eliminating an eye decreases the number of strands of the braid and
      therefore has to be done at most $\TheRank$ times. Note that
      during this process, the word length of the braid part in
      the triple does not increase by Observation~\ref{obs:cap_word_length}.

      Once the diagram is reduced, the braid part is put into
      right-greedy normal form in time
      \(
        \BigO{\TheHeight[][\Two]\TheRank\LogOf{\TheRank}}
        .
      \)
    \end{proof}

    \begin{prop}
      Let
      \(
        \TupelOf{
          \TheTree[\One][\Top],
          \TheWord[\One],
          \TheTree[\One][\Bot]
        }
      \)
      and
      \(
        \TupelOf{
          \TheTree[\Two][\Top],
          \TheWord[\Two],
          \TheTree[\Two][\Bot]
        }
      \)
      be two triples in normal form representing the elements
      $\TheBvElement[\One]$ and $\TheBvElement[\Two]$, respectively.
      Let
      $\TheRank[\One]$ and $\TheRank[\Two]$ be their numbers
      of strands and let $\TheHeight[\One]$ and $\TheHeight[\Two]$
      be the word lengths of $\TheWord[\One]$ and
      $\TheWord[\Two]$, respectively.

      The normal form triple representing the product
      $\TheBvElement[\One]\TheBvElement[\Two]$ can be computed
      in time
      \(
        \BigO{
          \Parentheses[][\Two]{
            \TheHeight[\One]+\TheHeight[\Two]
          }
          \Parentheses[][\Two]{
            \TheRank[\One]+\TheRank[\Two]
          }
          \LogOf{\TheRank[\One]+\TheRank[\Two]}
        }
        .
      \)
    \end{prop}
    \begin{proof}
      Using Algorithm~\ref{alg:product}, we have to control how the
      number of strands and the word length of the braid grow in
      the unreducing step. For either factor, the
      number of strands grows at most to
      $\TheRank[\One]+\TheRank[\Two]$ since $\TheTree[\Two][\Top]$
      has at most $\TheRank[\Two]$ carets that need to be cloned in
      $\TheTree[\One][\Bot]$ and $\TheTree[\One][\Bot]$ has at most
      $\TheRank[\One]$ carets that we might need to
      recreate in
      $\TheTree[\Two][\Top]$.
      Hence, we have to double at most
        $\TheRank[\Two]$ strands in $\TheWord[\One]$, which
        can be done in time
        \(
          \BigO{
            \TheRank[\Two]
            \TheHeight[\One]
            \Parentheses{
              \TheRank[\One]+\TheRank[\Two]
            }
            \LogOf{\TheRank[\One]+\TheRank[\Two]}
          }
          ;
        \)
        and we have do double at most
        $\TheRank[\One]$ strands in $\TheWord[\Two]$, which can be done
        in time
        \(
          \BigO{
            \TheRank[\One]
            \TheHeight[\Two]
            \Parentheses{
              \TheRank[\One]+\TheRank[\Two]
           }
            \LogOf{\TheRank[\One]+\TheRank[\Two]}
          }
          .
        \)
        The total time for unreducing the diagrams is therefore
        \(
          \BigO{
            \Parentheses{
              \TheHeight[\One]+\TheHeight[\Two]
            }
            \Parentheses[][\Two]{
              \TheRank[\One]+\TheRank[\Two]
            }
            \LogOf{\TheRank[\One]+\TheRank[\Two]}
          }
          .
        \)

      By Observation~\ref{obs:cap_word_length}, unreducing does not
      increase the word length of the braids. Thus, it follows
      from Proposition~\ref{prop:reduce} that we can reduce the
      triple that we obtain for the product
      $\TheBvElement[\One]\TheBvElement[\Two]$ to normal form
      in time
      \(
        \BigO{
            \Parentheses[][\Two]{
              \TheHeight[\One]+\TheHeight[\Two]
            }
            \Parentheses[][\Two]{
              \TheRank[\One]+\TheRank[\Two]
            }
            \LogOf{\TheRank[\One]+\TheRank[\Two]}
          },
      \)
      which dominates all other bounds.
    \end{proof}

    \begin{rem}
      On can save some computational effort by not putting all braids
      into normal form. Dropping the normalization steps from the
      algorithms above yields the following complexity bounds:
      \begin{enumerate}
        \item
          Any triple
          \(
            \TupelOf{
              \TheTree[][\Top],
              \TheWord,
              \TheTree[][\Bot]
            },
          \)
          where the trees have $\TheRank$ leaves and $\TheWord$ has
          length $\TheHeight$, can be reduced in time
          \(
            \BigO{
              \TheHeight\TheRank[][\Two]\LogOf{\TheRank}
            }
            .
          \)
        \item
          Let
          \(
            \TupelOf{
              \TheTree[\One][\Top],
              \TheWord[\One],
              \TheTree[\One][\Bot]
            }
          \)
          and
          \(
            \TupelOf{
              \TheTree[\Two][\Top],
              \TheWord[\Two],
              \TheTree[\Two][\Bot]
            }
          \)
          be two semi-reduced triples representing
          the elements $\TheBvElement[\One]$ and
          $\TheBvElement[\Two]$, respectively.
            For $\TheIndex\in\SetOf{\One,\Two}$, let
            $\TheRank[\TheIndex]$ be the number of strands
            in $\TheWord[\TheIndex]$ and let $\TheHeight[\TheIndex]$
            be the word length of $\TheWord[\TheIndex]$.
          A semi-reduced triple representing the product
          $\TheBvElement[\One]\TheBvElement[\Two]$ can be computed
          in time
          \(
            \BigO{
              \Parentheses{
                \TheHeight[\One]+\TheHeight[\Two]
              }
              \Parentheses{
                \TheRank[\One]+\TheRank[\Two]
              }
              \LogOf{
                \TheRank[\One]+\TheRank[\Two]
              }
            }
            .
          \)
          This triple has trees with at most
          $\TheRank[\One]+\TheRank[\Two]$
          leaves and a braid that is represented as a word of
          length $\TheHeight[\One]+\TheHeight[\Two]$.
      \end{enumerate}
      From a practical point of view, it therefore pays off to put
      elements into normal form only when one needs to
      test for equality.
    \end{rem}
    \begin{rem}
      We note that the bit-length needed to encode a triple with
      $\TheRank$ strands and the braid given as a word of length
      $\TheHeight$ is about $\TheHeight\TheRank\LogOf{\TheRank}$.
      Thus, multiplication of elements in $\BV$ is actually quadratic
      in terms of total length of inputs, i.e., multiplication in
      $\BV$ is about as efficient as the elementary school algorithm
      for multiplying multi-digit integers.
    \end{rem}

  \section{Linear Representations}\label{sec:representations}

    \begin{lemma}[{\cite[Corollary~4.14]{Brin:2007}}]
      The group $\BV$ is generated by three families of
      generators
      $\EffGen[\GenIndex]$,
      $\BrdGen[\GenIndex]$,
      and
      $\AltGen[\GenIndex]$
      (where $\GenIndex\geq\Zero$)
      subject to the following relations:
      \[
        \begin{array}{r@{\,=\,}l@{\kern1cm}l}
          \EffGen[\LowIndex]\EffGen[\HighIndex]
          &
          \EffGen[\HighIndex]\EffGen[\LowIndex+\One]
          &
          \HighIndex<\LowIndex
          \\
          \AltGen[\HighIndex][\Eps]
          \EffGen[\HighIndex]
          &
          \EffGen[\HighIndex+\One]
          \AltGen[\HighIndex][\Eps]
          \AltGen[\HighIndex+\One][\Eps]
          &
          \HighIndex\geq\Zero,
          \Eps = \pm\One
          \\
          \AltGen[\LowIndex]\EffGen[\HighIndex]
          &
          \EffGen[\HighIndex]\AltGen[\LowIndex]
          &
          \HighIndex>\LowIndex+\One
          \\
          \BrdGen[\LowIndex]\EffGen[\HighIndex]
          &
          \EffGen[\HighIndex]\BrdGen[\LowIndex+\One]
          &
          \HighIndex<\LowIndex
          \\
          \AltGen[\HighIndex]
          &
          \BrdGen[\HighIndex+\One][-\One]
          \EffGen[\HighIndex][-\One]
          \BrdGen[\HighIndex]
          &
          \HighIndex\geq\Zero
          \\
          \AltGen[\LowIndex]\AltGen[\HighIndex]
          &
          \AltGen[\HighIndex]\AltGen[\LowIndex]
          &
          \AbsValue{\HighIndex-\LowIndex}\geq\Two
          \\
          \AltGen[\HighIndex]
          \AltGen[\HighIndex+\One]
          \AltGen[\HighIndex]
          &
          \AltGen[\HighIndex+\One]
          \AltGen[\HighIndex]
          \AltGen[\HighIndex+\One]
          &
          \HighIndex\geq\Zero
          \\
          \BrdGen[\LowIndex]
          \AltGen[\HighIndex]
          &
          \AltGen[\HighIndex]
          \BrdGen[\LowIndex]
          &
          \LowIndex\geq\HighIndex+\Two
          \\
          \AltGen[\HighIndex]
          \BrdGen[\HighIndex+\One]
          \AltGen[\HighIndex]
          &
          \BrdGen[\HighIndex+\One]
          \AltGen[\HighIndex]
          \BrdGen[\HighIndex+\One]
          &
          \HighIndex\geq\Zero
          \\
          \AltGen[\GenIndex]
          &
          \BrdGen[\GenIndex]
          \EffGen[\GenIndex]
          \BrdGen[\GenIndex+\One][-\One]
          &
          \GenIndex \geq \Zero
        \end{array}
      \]
      Moreover,
      \begin{enumerate}
        \item
          The family
          $\SetOf[{\EffGen[\GenIndex]}]{\GenIndex\geq\Zero}$
          generates a copy of $\ThEff$ inside $\BV$.
        \item
          Imposing the additional relations
          \[
            \BrdGen[\GenIndex][\Two]
            =
            \AltGen[\GenIndex][\Two]
            =
            \TheTrivialElement,
            \qquad
            \GenIndex\geq\Zero
          \]
          turns the above into a presentation for $\ThVee$.
      \end{enumerate}
    \end{lemma}
    In particular, $\ThVee$ is a quotient of $\BV$. Thus, $\BV$ is
    not simple. We shall show, however, that it is not too far from
    being simple: the normal closure of $[\ThEff,\ThEff]$ (regarded
    as a subgroup of $\BV$) is all of $\BV$:

    \begin{lemma}\label{lemma:commutators}
      Consider $\ThEff$ as a subgroup of $\BV$, generated by
      $\SetOf[{\EffGen[\GenIndex]}]{ \GenIndex\geq\Zero}$.
      Then, $\BV$ does not have a proper normal subgroup
      containing
      $[\ThEff,\ThEff]$.
    \end{lemma}
    \begin{proof}
      We first note that for $\TheIndex\geq\One$,
      \[
        \EffGen[\TheIndex]\EffGen[\TheIndex+\One][-\One]
        =
        \EffGen[\Zero]\EffGen[\TheIndex+\One]\EffGen[\Zero][-\One]
        \EffGen[\TheIndex+\One][-\One]
        =
        [\EffGen[\Zero],\EffGen[\TheIndex+\One]]
      \]
      and
      \[
        \EffGen[\TheIndex+\One]\EffGen[\TheIndex][-\One]
        =
        \EffGen[\TheIndex+\One]
        \EffGen[\Zero]\EffGen[\TheIndex+\One][-\One]\EffGen[\Zero][-\One]
        =
        [\EffGen[\TheIndex+\One],\EffGen[\Zero]]
      \]
      are commutators. Telescoping products of such commutators shows that
      \(
        \EffGen[\TheIndex]\EffGen[\AltIndex][-\One]\in
        [\ThEff,\ThEff]
      \)
      for $\TheIndex,\AltIndex\geq\One$.

      Let $\TheNormalClosure$ be the normal closure of
      $[\ThEff,\ThEff]$ in $\BV$. For all $\TheIndex\geq\One$,

        \[
          \TheNormalClosure\ni
          \AltGen[\TheIndex]\EffGen[\TheIndex]\EffGen[\TheIndex+\Two][-\One]
          \AltGen[\TheIndex][-\One]
          =
          \AltGen[\TheIndex]\EffGen[\TheIndex]\AltGen[\TheIndex][-\One]
          \EffGen[\TheIndex+\Two][-\One]
          =
          \EffGen[\TheIndex+\One]\AltGen[\TheIndex]\AltGen[\TheIndex+\One]
          \AltGen[\TheIndex][-\One]
          \EffGen[\TheIndex+\Two][-\One]
          .
        \]

      Hence,
      \(
        \AltGen[\TheIndex]\AltGen[\TheIndex+\One]\AltGen[\TheIndex][-\One]
        \EffGen[\TheIndex+\Two][-\One]\EffGen[\TheIndex+\One]
        \in
        \TheNormalClosure
        ,
      \)
      and therefore
      \(
        \AltGen[\TheIndex]\AltGen[\TheIndex+\One]\AltGen[\TheIndex][-\One]
        \in
        \TheNormalClosure.
      \)
      Thus, $\AltGen[\TheIndex+\One]\in\TheNormalClosure$ for each
      $\TheIndex\geq\One$.

      Now, we show that all generators of $\BV$ die in the quotient
      $\BV\rmod\TheNormalClosure$. We already know this
      for $\AltGen[\TheIndex]$ with $\TheIndex\geq\Two$.
      Using the braid relations between $\AltGen[\One]$ and
      $\AltGen[\Two]$, we find that $\AltGen[\One]$ dies as well,
      and then, in view of the braid relation between
      $\AltGen[\Zero]$ and $\AltGen[\One]$, we find that
      $\AltGen[\Zero]$ dies as well.

      The family of mixed braid relations (between
      $\AltGen[\TheIndex]$ and $\BrdGen[\TheIndex+\One]$) now
      implies that $\BrdGen[\TheIndex]=\TheTrivialElement$
      in $\BV\rmod\TheNormalClosure$ for $\TheIndex\geq\One$.
      Now the relations
      \(
        \AltGen[\Zero]=
        \BrdGen[\Zero]\EffGen[\Zero]\BrdGen[\One][-\One]
      \)
      and
      \(
        \AltGen[\Zero]=
        \BrdGen[\One][-\One]\EffGen[\Zero][-\One]\BrdGen[\Zero]
      \)
      imply
      \(
        \AltGen[\Zero][\Two]=\TheTrivialElement
      \)
      in $\BV\rmod\TheNormalClosure$.

      Thus, the squares of all $\AltGen[\TheIndex]$
      and all $\BrdGen[\TheIndex]$ die in
      $\BV\rmod\TheNormalClosure$, whence $\BV\rmod\TheNormalClosure$
      is a quotient of $\ThVee$. However, already too many generators
      are gone. So $\BV\rmod\TheNormalClosure$ is a proper quotient
      of $\ThVee$, and therefore trivial.
    \end{proof}

    \begin{observation}
      Any linear representation of a simple group is either faithful
      or trivial.\qed
    \end{observation}
    \begin{cor}
      Neither the commutator subgroup $[\ThEff,\ThEff]$ in
      Thompson's group $\ThEff$ nor Thompson's group $\ThVee$ do
      admit a non-trivial linear representation (in any characteristic).
    \end{cor}
    \begin{proof}
      First note that $\ThEff$ is not linear in any characteristic:
      it is finitely generated and not solvable. It it
      was linear, it would contain a non-abelian free subgroup by the
      Tits Alternative. But $\ThEff$ does not contain non-abelian free
      subgroups.

      The commutator subgroup $[\ThEff,\ThEff]$ is also not linear
      in any characteristic since it contains a copy of $\ThEff$ as
      a subgroup. The claim for $[\ThEff,\ThEff]$ nor follows since
      $[\ThEff,\ThEff]$ is simple.

      The same argument applies to Thompson's group $\ThVee$, which
      is simple and also contains a copy of $\ThEff$.
    \end{proof}

    The main theorem now follows immediately:
    \begin{theorem}\label{thm:no_representations}
      The group $\BV$ does not admit non-trivial linear representations
      in any characteristic.
    \end{theorem}
    \begin{proof}
      The subgroup
      $[\ThEff,\ThEff]$ lies within the kernel of any linear representation
      of $\BV$. However, such a kernel is a normal subgroup and
      therefore exhausts $\BV$ by Lemma~\ref{lemma:commutators}.

    \end{proof}

  \section{On the Cryptographic Use
  of \boldmath$\BV$}\label{sec:cryptography}

    After the paper by Anshel, Anshel, and Goldfeld
    \cite{Anshel.Anshel.Goldfeld:1999}, group based
    cryptography got a huge boost and is rapidly developing
    since. An idea behind using groups in cryptography is
    that finding solutions of certain equations or systems
    of equations over a given group is computationally infeasible
    while generating equations with known or given solutions might
    be efficient since it only involves
    multiplication and computing normal forms.

    We recall the key-exchange protocol proposed by
    Anshel, Anshel, and Goldfeld. Below,
    \(
      \AliceSize, \AliceLength, \BobSize,
    \)
    and
    \(
      \BobLength
    \)
    are integer parameters and $\TheGroup$ is a group,
    called the \notion{platform group} of the protocol.
    A key-exchange has the goal that Alice and Bob collaboratively
    create a secret that is shared between them. In this particular
    protocol, the shared secret will be an element of
    $\TheGroup$. It is selected as follows:
    \begin{enumerate}
      \item
        Alice chooses randomly a public set
        \(
          \SetOf{
            \AliceElement[\One],\ldots,\AliceElement[\AliceSize]
          }
          \subset\TheGroup
        \)
        and a private key
        \(
          \AliceKey =
          \AliceElement[{\AliceIndex[\One]}][{\AliceExp[\One]}]
          \cdots
          \AliceElement[{\AliceIndex[\AliceLength]}][{\AliceExp[\AliceLength]}]
          \in
          \GroupPresented{
            \AliceElement[\One],\ldots,\AliceElement[\AliceSize]
          }
          \subseteq\TheGroup.
        \)
      \item
        Bob chooses randomly a public set
        \(
          \SetOf{
            \BobElement[\One],\ldots,\BobElement[\BobSize]
          }
          \subset\TheGroup
        \)
        and a private key
        \(
          \BobKey =
          \BobElement[{\BobIndex[\One]}][{\BobExp[\One]}]
          \cdots
          \BobElement[{\BobIndex[\BobLength]}][{\BobExp[\BobLength]}]
          \in
          \GroupPresented{
            \BobElement[\One],\ldots,\BobElement[\BobSize]
          }
          \subseteq\TheGroup.
        \)
      \item
        Alice sends to Bob the $\BobSize$-tuple
        \(
          \SetOf{
            \AliceKey\BobElement[\One]\AliceKey[][-\One],\ldots,
            \AliceKey\BobElement[\BobSize]\AliceKey[][-\One]
          }.
        \)
      \item
        Bob sends to Alice the $\AliceSize$-tuple
        \(
          \SetOf{
            \BobKey\AliceElement[\One]\BobKey[][-\One],
            \ldots,
            \BobKey\AliceElement[\AliceSize]\BobKey[][-\One]
          }.
        \)
      \item
        The shared secret is the commutator
        \(
          [\AliceKey,\BobKey]
          =
          \AliceKey[][-\One]\BobKey[][-\One]\AliceKey\BobKey
          ,
        \)
        which both of them can compute.
    \end{enumerate}

    The security of this key-exchange protocol depends
    on how hard it is to solve the Simultaneous Conjugacy Search
    Problem in $\TheGroup$: given elements
    \(
      \TheGroupElement[\One],\ldots,\TheGroupElement[\TheLastIndex]
    \)
    and
    \(
      \AltGroupElement[\One],\ldots,\AltGroupElement[\TheLastIndex]
    \)
    in $\TheGroup$, find an element
    $\TheConjugator\in\TheGroup$ such that
    \(
      \TheGroupElement[\TheIndex]
      =
      \TheConjugator[][-\One] \AltGroupElement[\TheIndex]
      \TheConjugator
    \)
    provided it is known that such a conjugating element exists.

    Certain criteria on the choice of the platform group for a
    cryptosystem were given by Shpilrain \cite{Shpilrain:2004}.
    We note that $\BV$ satisfies those criteria. In
    Section~\ref{sec:complexity}, we have shown that
    computations in $\BV$ can be performed in polynomial time
    and that the word problem can also be solved in polynomial
    time. The group $\BV$ has a presentation with many short
    relations \cite{Brin:2006}. According to \cite{Shpilrain:2004},
    this might make it harder to mount length based attacks
    on $\BV$ (more on this below). Finally, both braid groups and
    Thompson's groups $\ThVee$ and $\ThEff$ are widely known,
    which makes the braided version $\BV$ ``marketable''.

    Both, braid groups and Thompson's
    group $\ThEff$ were investigated in the
    context of cryptography,
    see
      \cite{Dehornoy:2004},
      \cite{Mahlburg:2004},
      \cite{Shpilrain.Ushakov:2005}
    and references therein.
    In the remainder of this section, we shall compare $\BV$ to
    $\ThEff$ and the braid groups from a cryptographic point of
    view.

    The simultaneous conjugacy problem in $\ThEff$
    was solved by Kassabov and Matucci \cite{Kassabov.Matucci:2006}
    using the interpretation
    of elements of $\ThEff$ as piecewise linear functions. Such
    interpretation is not available for $\BV$.%

    The conjugacy search problem seems to be harder for $\BV$ than
    for braid groups. Efficient algorithms for solving the conjugacy
    problem in braid groups are based on associating a finite set
    (called summit set \cite{Garside:1969}, super summit set,
    and ultra summit set \cite{Elrifai.Morton:1994},
    \cite{Gebhardt:2005}) of conjugates to any braid
    $\TheBraid\in\BraidGroup[\TheRank]$. One should note that
    finiteness of the summit sets relies on the
    number of strands $\TheRank$ being fixed. Braids extracted
    from elements in $\BV$ can have an arbitrary number of strands,
    which makes it impossible to directly transfer to $\BV$
    strategies successful for braid groups.

    There are also known attacks on braid-group based crypto-systems
    using linear representations. Braid groups are known to be
    linear (\cite{Bigelow:2001}, \cite{Krammer:2002}), but more
    importantly, the Burau and colored Burau representations have
    small kernels and can be exploited. According to
    Theorem~\ref{thm:no_representations}, such attacks on $\BV$
    will not work.

    A very general approach, now known as \notion{length based attack},
    was described in \cite{Hughes.Tannenbaum:2002} and further
    developed in \cite{Garber.Kaplan.Teichner.Tsaban.Vishne:2006}.
    It relies on the existence of a good length function on the
    platform group, and can be used to solve arbitrary systems of
    equations over the group. The main idea is to use the length
    function to turn the system of equations into a problem in
    combinatorial optimization. We refer to
    \cite{Garber.Kaplan.Teichner.Tsaban.Vishne:2006},
    \cite{Ruinskiy.Shamir.Tsaban:2006}, and
    \cite{Myasnikov.Ushakov:2007a}
    for descriptions of length based attacks for the
    conjugacy search problem in braid groups and Thompson's group
    $\ThEff$. Length based attacks are most successful if randomly
    chosen subgroups of the platform group are generically free
    (see \cite{Myasnikov.Ushakov:2007b} for a detailed analysis).
    This is the case for braid groups
    \cite{Myasnikov.Osin:xxxx}. Both groups, $\ThVee$ and
    $\BV$ are known to have free subgroups. It is not known whether
    random subgroups of $\ThVee$ and $\BV$ are generically free.
    Thus, answers to the following questions will have an impact
    on the usability of $\BV$ for cryptography:
    \begin{question}
      What are generic subgroups of $\ThVee$ and $\BV$?
    \end{question}
    \begin{question}
      Does $\BV$ have a quotient with generically free subgroups?
    \end{question}

\end{document}